\documentclass[a4paper, 12pt, onecolumn]{article} \textwidth 148mm
\usepackage{amsmath}
\usepackage{amsthm}
\usepackage{amsfonts}
\usepackage{bbm}
\usepackage{CJK}
\usepackage{fancyhdr}
\usepackage{graphicx}
\usepackage{geometry}
\usepackage{psfrag}
\usepackage{amsfonts,amsmath,amsthm, amssymb}
\usepackage{latexsym, euscript, epic, eepic}
\usepackage{time}
\usepackage{txfonts}
\usepackage{colortbl}
\usepackage{stmaryrd}
\usepackage{mathrsfs}
\usepackage{txfonts}
\usepackage{amsfonts}
\usepackage{color}
\usepackage{lineno}

\usepackage{indentfirst,latexsym,bm}

 \numberwithin{equation}{section}

\begin{document}
\newtheorem{theorem}{Theorem}[section]
\newtheorem{proposition}[theorem]{Proposition}
\newtheorem{remark}[theorem]{Remark}
\newtheorem{corollary}[theorem]{Corollary}
\newtheorem{definition}{Definition}[section]
\newtheorem{lemma}[theorem]{Lemma}

\title{\large Geodesics in the Engel group with a sub-Lorentzian metric
\thanks{Supported by NSFC (No.11071119, No.11401531); NSFC-RFBR (No. 11311120055).}}
\author{\normalsize Qihui Cai$^1$, Tiren Huang$^2$\footnote{Corresponding author}, Yu. L. Sachkov$^3$,
Xiaoping Yang$^1$\\
\scriptsize 1 Department of Applied Mathematics, Nanjing University
of Science $\&$ Technology, Nanjing 210094,China\\
\scriptsize 2 Department of Mathematics, Zhejiang Sci-Tech
University, Hangzhou 310018, China\\
\scriptsize 3 Program Systems Institute Pereslavl-Zalessky 152140,
Russia\\
\scriptsize E-mail: caigada@gmail.com, htiren@ustc.edu.cn,
sachkov@sys.botik.ru, yangxp@njust.edu.cn.}
\date{}
\maketitle
 \fontsize{12}{22}\selectfont\small
\paragraph{Abstract:}
Let $E$ be the Engel  group  and $D$ be a rank 2 bracket generating
left invariant distribution with a Lorentzian metric, which is a
nondegenerate metric of index 1. In this paper, we first study some
properties of horizontal curves on $E$. Second, we prove that
time-like normal geodesics are locally maximizers in the Engel
group, and calculate the explicit expression of non-space-like
geodesics.
\\[10pt]
\\[10pt]

\vspace{1cm}\fontsize{12}{22}\selectfont

\section{\normalsize Introduction}
A sub-Riemannian structure on a manifold $M$ is given by a smoothly
varying distribution $D$ on $M$ and a smoothly varying positively
definite metric $g$ on the distribution. The triple $(M,D,g)$ is
called a \emph{sub-Riemannian manifold}, which has been applied in
control theory, quantum physics, C-R geometry and the other areas.
Some efforts have been made to generalize sub-Riemannian manifold.
One of them leads to the following question: what kind of
geometrical features the mentioned triple will have if we change the
positively definite metric to an indefinite nondegenerate metric? It
is natural to start with the Lorentzian metric of index 1. In this
case the triple: manifold, distribution and Lorentzian metric on the
distribution is called a \emph{sub-Lorentzian manifold} by analogy
with a Lorentzian manifold. For the details concerning the
\emph{sub-Lorentzian geometry}, the reader is referred to
\cite{M.Golubitsky3}.  To our knowledge, there are only a few works
devoted to this subject (see
\cite{Chang,M.Golubitsky3,M.Golubitsky4,
M.Golubitsky5,M.Golubitsky6,Korolko.A}). In \cite{Chang}, Chang,
Markina, and Vasiliev have systematically studied the geodesics in
an anti-de Sitter space with a sub-Lorentzian metric and a
sub-Riemannian metric respectively. In \cite{M.Golubitsky5},
Grochowski computed reachable sets starting from a point in the
Heisenberg sub-Lorentzian manifold on $\mathbb{R}^3$. It was shown
in \cite{Korolko.A} that the Heisenberg group $\mathbb{H}$ with a
Lorentzian metric on $\mathbb{R}^3$ possesses the uniqueness of
Hamiltonian geodesics of time-like or space-like type.

The Engel group was first named by Cartan \cite{Cartan} in 1901. It
is a prolongation of a three dimensional contact manifold, and is a
Goursat manifold. In \cite{Sachkov,Sachkov2,Sachkov3}, A.Ardentov
and Yu.L.Sachkov computed minimizers on the sub-Riemannian Engel
group. In the present article, we study the Engel group furnished
with a sub-Lorentzian metric. This is an interesting example of
sub-Lorentzian manifolds, because the Engel group is the simplest
manifold with nontrivial abnormal extremal trajectories, and the
vector distribution of the Engel group is not $2-$ generating, its
growth vector is $(2,3,4)$. We first study some properties of
horizontal curves in the Engel group. Second, we use the Hamiltonian
formalism and Pontryagin maximum principle to write the equations
for geodesics. Furthermore, we give a complete description of the
Hamiltonian geodesics in the Engel group.

Apart from the introduction, this paper contains three sections.
Section 2 contains some preliminaries as well as definitions of
sub-Lorentzian manifolds, the Engel group. In Section 3, we study
some properties of horizontal curves in the Engel group. In Section
4, we prove that the time-like normal geodesics are locally maximal
in the Engel group , and explicitly calculate the non-space-like
Hamiltonian geodesics.

\section{\normalsize
Preliminaries}\label{sec:bd} A sub-Lorentzian manifold is a triple
$(M,D,g)$, where $M$ is a smooth $n$-dimensional manifold, $D$ is a
smooth distribution on $M$ and $g$ is a smoothly varying Lorentzian
metric on $D$. For each point $p\in M$, a vector $v\in D_p$ is said
to be horizontal. An absolutely continuous curve $\gamma(t)$ is said
to be horizontal if its derivative $\gamma'(t)$ exists almost
everywhere and lies in $D_{\gamma(t)}$.

A vector $v\in D_p$ is said to be time-like if $g(v,v)<0$;
space-like if $g(v,v)>0$ or $v=0$; null(light-like) if $g(v,v)=0$
and $v\neq 0$; and non-space-like if $g(v,v)\leq 0$. A curve
$\gamma(t)$ is said to be time-like if its tangent vector
$\dot\gamma(t)$ is time-like a.e.; space-like if $\dot\gamma(t)$ is
space-like a.e.; null if $\dot\gamma(t)$ is null a.e.;
non-space-like if $\dot\gamma(t)$ is non-space-like a.e..

By a time orientation of $(M,D,g)$, we mean a continuous time-like
vector field on $M$. From now on, we assume that $(M,D,g)$ is
time-oriented. If $X$ is a time orientation on $(M,D,g)$, then a
non-space-like vector $v\in D_p$ is said to be future directed if
$g(v,X(p))<0$, and past directed if $g(v,X(p))>0$. Throughout this
paper, ``f.d." stands for ``future directed", ``t." for
``time-like", and ``nspc." for ``non-space-like".

Let $v,w\in D$ be two non-space-like vectors, we have the following
reverse Schwartz inequality (see page 144 in \cite{O'Neill}):
$$|g(v,w)|\geq\|v\|\cdot\|w\|,$$
where $\|v\|=\sqrt{|g(v,v)|}$. The equality holds if and only if $v$
and $w$ are linearly dependent.

We introduce the space $H_{\gamma(t)}$ of horizontal nspc.\ curves:
\begin{align}
\nonumber H_{\gamma(t)}&=\{\gamma:[0,1]\rightarrow M|\ \gamma(t)
\hbox{ is
absolutely continuous },\ \ g(\dot{\gamma}(t),\dot{\gamma}(t))\leq0,\\
 &\ \ \ \ \ \
  \dot{\gamma}(t)\in D_{\gamma(t)} \hbox{ for almost all }t \in
[0,1]\}.
\end{align}
The sub-Lorentzian length of a horizontal nspc.\ curve $\gamma(t)$
is defined as follows:
$$l(\gamma)=\int_0^1 \|\gamma'(t)\|dt,$$
where $\|\gamma'(t)\|=\sqrt{|g(\gamma'(t),\gamma'(t))|}.$ We use the
length to define the sub-Lorentzian distance $d_U(q_1,q_2)$ with
respect to a set $U\subset M$ between two points $q_1,q_2\in U$:

\[ d_U(q_1,q_2)=\bigg\{\begin{array}{ll} \sup\{l(\gamma), \gamma\in
H_U(q_1,q_2)\} & \hbox{if} \  \ H_U(q_1,q_2)\neq \emptyset \\ 0
&\hbox{if}\  \ H_U(q_1,q_2)=\emptyset,
\end{array}
\]
where $H_U(q_1,q_2)$ is the set of all nspc.f.d\ curves contained in
$U$ and joining $q_1$ and $q_2$.

A nspc.\ curve is said to be a maximizer if it realizes the distance
between its endpoints. We also use the name $U$-geodesic for a curve
in $U$ whose each suitably short sub-arc is a $U$-maximizer.

A distribution $D\subset TM$ is called bracket generating if any
local frame $\{X_i\}_{1\leq i\leq r}$ for $D$, together with all of
its iterated Lie brackets $[X_i,X_j],[X_i,[X_j,X_k]],\cdots$ span
the tangent bundle $TM$. Bracket generating distributions are
sometimes also called completely nonholonomic distributions, or
distributions satisfying H$\ddot{o}$rmander's condition.

\begin{theorem}(Chow) Fix a point $q\in M$. If the distribution $D\subset TM$ is bracket
generating then the set of points that can be connected to $q$ by a
horizontal curve is the component of $M$ containing $q$.
\end{theorem}

By Chow's Theorem, we know that if $D$ is bracket generating and $M$
is connected, then any two points of $M$ can be joined by a
horizontal curve.

Now, we describe the Engel group $E$. We consider the Engel group
$E$ with coordinates $q=(x_1,x_2,y,z)\in \mathbb{R}^4$. The group
law is denoted by $\odot$ and defined as follows:
\begin{align}
&(x_1,x_2,y,z)\odot(x_1',x_2',y',z')\nonumber\\
&=\left(x_1+x_1',x_2+x_2',y+y'+\frac{x_1x_2'-x_1'x_2}{2},z+z'+\frac{x_2x_2'}{2}(x_2+x_2')+x_1y'+\frac{x_1x_2'}{2}(x_1+x_1')\right).\nonumber
\end{align}

A vector field $X$ is said to be left-invariant if it satisfies
$dL_qX(e)=X(q)$, where $L_q$ denotes the left translation
$p\rightarrow L_q(p)=q\odot p$ and $e$ is the identity of $E$. This
definition implies that any left-invariant vector field on $E$ is a
linear combination of the following vector fields:
\begin{align}\label{frame}
&X_1=\frac{\partial}{\partial x_1}-\frac {x_2}{2}
\frac{\partial}{\partial y};\ \  X_2=\frac{\partial}{\partial
x_2}+\frac {x_1} {2}\frac{\partial}{\partial
y}+\frac{x_1^2+x_2^2}{2}\frac{\partial}{\partial z};\nonumber\\
&X_3=\frac{\partial}{\partial y}+x_1\frac{\partial}{\partial z};\ \
\ \  X_4=\frac{\partial}{\partial z}.
\end{align}
The distribution $D=span\{X_1,X_2\}$ of $E$ satisfies the bracket
generating condition, since $ X_3=[X_1,X_2], X_4=[X_1,X_3]$.  The
Engel group is a nilpotent Lie group, since $
[X_1,X_4]=[X_2,X_3]=[X_2,X_4]=0.$ We define a smooth Lorentzian
metric $\tilde{g}$ on $E$ such that
$\tilde{g}(X_i,X_j)=(-1)^{\delta_{1i}}\delta_{ij}$, $i,j=1,\cdots,
4$, where $\delta_{ij}$ is the Kronecker symbol. It is not difficult
to compute the coefficients of $\tilde{g}$ under the local
coordinates $(x_1,x_2,y,z)\in \mathbb{R}^4$. The coefficients can be
expressed as
\begin{equation}\label{coe
matrix} (\tilde{g}_{ij})=\begin{pmatrix}
 -1+\frac{x_{2}^{2}}{4}+\frac{x_{1}^{2}x_{2}^{2}}{4}
  & -\frac{x_{1}x_{2}}{4}+\frac{x_{1}x_{2}^{3}}{4}
  & \frac{x_{2}}{2}+\frac{x_{2}x_{1}^{2}}{2}
  & -\frac{x_{1}x_{2}}{2} \\
-\frac{x_{1}x_{2}}{4}+\frac{x_{1}x_{2}^{3}}{4}
  & 1+\frac{x_{1}^{2}}{4}+\frac{x_{2}^{4}}{4}
  & -\frac{x_{1}}{2}+\frac{x_{1}x_{2}^{2}}{2}
  & -\frac{x_{2}^{2}}{2} \\
\frac{x_{2}}{2}+\frac{x_{2}x_{1}^{2}}{2}
  & -\frac{x_{1}}{2}+\frac{x_{1}x_{2}^{2}}{2}
  & 1+x_{1}^{2}
  & -x_{1} \\
-\frac{x_{1}x_{2}}{2}
  & -\frac{x_{2}^{2}}{2}
  & -x_{1}
  & 1
\end{pmatrix}
\end{equation}
When we restrict $\tilde{g}$ to $D$, we can get a smooth sub-Lorentzian
metric $g=\tilde{g}_D$, which satisfies
\begin{align}\label{metric}
g(X_1,X_1)=-1,\ \  g(X_2,X_2)=1, \ \ g(X_1,X_2)=0.
\end{align}
On the other hand, any sub-Lorentzian metric on $D$ can be extended
to a (usually not unique) Lorentzian metric on $E$. In this paper,
we assume that $X_1$ is the time orientation.

\section{\normalsize Horizontal curves}
Chow's theorem states that any two points can be connected by a
horizontal curve, but we have no information about the character of
horizontal curves. In this section, we will investigate some
properties of horizontal curves.

An absolutely continuous curve $\gamma(s):[0,1]\to E$ is said to be
horizontal if the tangent vector $\dot{\gamma}(s)$ can be expressed
linearly by the horizontal directions $X_{1}$, $X_{2}$, hence we
have the following lemma.

\begin{lemma}A curve $\gamma(s)=(x_{1}(s),x_{2}(s),y(s),z(s))$ is
horizontal with respect to the distribution $D$, if and only if
\begin{align}\label{h condition}
&\frac{x_2\dot{x_1}}{2}-\frac{x_{1}\dot{x_{2}}}{2}+\dot{y}=0,\nonumber\\
&-\frac{x_1^2+x_2^2}{2}\dot x_2+\dot{z}=0.
\end{align}
\end{lemma}
\begin{proof}
The distribution $D$ is the annihilator of the one-forms:
$$\omega_1=\frac{x_2}{2}dx_1-\frac{x_1}{2}dx_2+dy,\ \ \ \omega_2=-\frac{x_1^2+x_2^2}{2}dx_2+dz$$
so $\gamma(s)$ is horizontal if and only if (\ref{h condition})
holds.
\end{proof}

By the same method, we can easily calculate the left invariant
coordinates $u_{1}(s)$ and $u_{2}(s)$ of the horizontal curve
$\gamma(s)$: \begin{align}\label{h coef1} u_{1}=\dot{x_{1}},\ \
u_{2}=\dot{x_{2}}.
\end{align}
The square of the velocity vector for the horizontal curve is:
\begin{align}\label{t condition}
g(\dot\gamma,\dot\gamma)=-u_{1}^{2}+u_{2}^{2}=-\dot{x}_{1}^{2}+\dot{x}_{2}^{2}.
\end{align}
So whether a horizontal curve is time-like(or nspc.) is determined
by the sign of $-\dot{x}_{1}^{2}+\dot{x}_{2}^{2}$.

Next we present a left invariant property of horizontal curves in
the Lie group with sub-Lorentzian metric. That is to say, the causal
character (time-like, space-like, light-like, or non-space-like) of
horizontal curves will not change under left translations. Hence it
is also true for the Engel group.

Let us consider a left-invariant sub-Lorentzian structure on a Lie
group $G$: $\mathscr{D}=span(X_{1},X_{2},\cdots,X_{k})\subset TG$,\
\ $g(X_{i},X_{j})=(-1)^{\delta_{1i}}\delta_{ij}$, with a time
orientation $X_{1}$. The vector fields $X_{i}$ are left-invariant,
i.e.
$$L_{x*}X_{i}(q)=X_{i}(x\cdot q),\ \ x,q\in G, \ \ i=1,\cdots, k.$$

\begin{proposition}\label{left}Left translations preserve the causal character
of horizontal curves of a left-invariant sub-Lorentzian structure on
a Lie group $G$, and the property of future-directness is also
preserved.
\end{proposition}

\begin{proof}
Let $c(t)$ be a causal horizontal curve, and
\begin{equation}
\dot{c}(t)=\sum_{i=1}^{k}u_{i}(t)X_{i}(c(t)).\nonumber
\end{equation}
Then, the left translation $\gamma(t)=x\odot c(t)$ has the same
causal character, since
\begin{align}
\dot{\gamma}(t)&=L_{x*}\dot{c}(t)=L_{x*}(\sum_{i=1}^{k}u_{i}(t)X_{i}(c(t)))
               =\sum_{i=1}^{k}u_{i}(t)L_{x*}(X_{i}(c(t)))\nonumber\\
               &=\sum_{i=1}^{k}u_{i}(t)X_{i}(x\odot
               c(t))=\sum_{i=1}^{k}u_{i}(t)X_{i}(\gamma(t)).\nonumber
\end{align}
Therefore,
\begin{align}
&g(\dot{c}(t),\dot{c}(t))=\sum_{i=1}^{k}(-1)^{\delta_{i1}}u_{i}^{2}=g(\dot{\gamma}(t),\dot{\gamma}(t)),\nonumber\\
&g(\dot{c}(t),X_{1})=-u_{1}=g(\dot{\gamma}(t),X_{1}).\nonumber
\end{align}

\end{proof}
By Chow's Theorem, we know that any two points on the Engel group
can be connected by a horizontal curve. But we do not know its
causal character(time-likeness, space-likeness, light-likeness).
This is not an easy problem. We are able to present some particular
examples to show its complexity.

\textbf{Example 1}: Let $\dot{x}_{2}=0$. Then $x_{2}=x_{2}^{0}$ is
constant. The horizontal condition (\ref{h condition}) becomes
\begin{align}
&\frac{x_{2}}{2}\dot{x_{1}}+\dot{y}=0,\label{th 1}\\
&\dot{z}=0.\label{th 2}
\end{align}
And the square of the velocity vector
\begin{align}\label{tt condition}
-u_{1}^{2}+u_{2}^{2}=-\dot{x}_{1}^{2}\leq0.
\end{align}
It follows that, the curves satisfying (\ref{th 1}) and (\ref{th 2})
are all non-space-like curves. Furthermore, we obtain,
\begin{align}\label{t1}
y(s)=-\frac{1}{2}x_{2}^{0}x_{1}(s)+\frac{1}{2}x_{2}^{0}x_{1}^{0}+y^{0},\
\ \ \ z(s)=z^0.
\end{align}
Therefore, all nonconstant horizontal curves
$c(s)=(x_{1}(s),x_{2}^{0},-\frac{x_{1}(s)x_{2}^{0}}{2}+\frac{x_{1}^{0}x_{2}^{0}}{2}+y^{0},z^{0})$
are time-like. These curves are straight lines. If $\dot x_1=0$,
$c(s)$ degenerate into some points, so there are no null curves in
this family.

\textbf{Example 2}: Let $\dot{x}_{2}\neq0$. We choose $x_{2}$ as a
parameter, then the horizontal condition (\ref{h condition}) becomes
\begin{align}
&\frac{x_{2}}{2}\dot{x_{1}}-\frac{x_{1}}{2}+\dot{y}=0,\label{th 3}\\
&-\frac{x_1^2+x_2^2}{2}+\dot{z}=0.\label{th 4}
\end{align}
And the square of the velocity vector
\begin{align}\label{tt2 condition}
-u_{1}^{2}+u_{2}^{2}=-\dot{x}_{1}^{2}+1.
\end{align}
We consider there different cases.\\
(a)~If $\dot{x}_{1}=0$, then $x_{1}=x_{1}^0$ is constant, (\ref{th
3}) and (\ref{th 4}) become
\begin{align}
&-\frac{x_{1}^0}{2}+\dot{y}=0,\label{th 11}\\
&-\frac{\left(x_1^0\right)^2+x_2^2}{2}+\dot{z}=0\label{th 12}.
\end{align}
In this case, $|\dot{c}(s)|^{2}=1$, so the curves satisfying
(\ref{th 11}) and (\ref{th 12}) are all space-like. Furthermore, we
obtain,
\begin{align}\label{t1}
y(s)=\frac{x_{1}^0}{2}x_2+y^0,\ \ \ \
z(s)=\frac{1}{6}x_2^3+\frac{\left(x_1^0\right)^2}{2}x_2+z_0.
\end{align}
Therefore, all nonconstant horizontal curves
$c(s)=(x_{1}^0,x_{2},\frac{x_{1}^0}{2}x_2+y^0,\frac{1}{6}x_2^3+\frac{\left(x_1^0\right)^2}{2}x_2+z_0)$
are space-like. There are no
null or time-like horizontal curves in this family.\\
(b)~If $\dot{y}=0$, (\ref{th 3}) and (\ref{th 4}) become
\begin{align}
&x_{2}\dot{x_{1}}-x_{1}=0,\label{th 5}\\
&-\frac{x_1^2+x_2^2}{2}+\dot{z}=0.\label{th 6}
\end{align}
From (\ref{th 5}), we get
$$\frac{1}{x_{2}}=\frac{\dot{x}_{1}}{x_{1}},$$
integrating with respect to $x_{2}$, we calculate $x_{1}=\iota
x_{2}$, where $\iota=\frac{x_{1}^{0}}{x_{2}^{0}}$, i.e.
$x_{1}=\frac{x_{1}^{0}}{x_{2}^{0}}x_{2}$, substituting $x_{1}$ in
(\ref{th 6}), we obtain
\begin{align}
z=\frac{1}{6}\left(1+\iota^{2}\right)x_{2}^{3}+z^{0}.
\end{align}
Therefore, all nonconstant horizontal curves
\begin{align}
c(s)=\left(\iota
x_{2},x_{2},y^{0},\frac{1}{6}\left(1+\iota^{2}\right)x_{2}^{3}+z^{0}\right)\label{h
c}
\end{align}
are time-like when $|\iota|>1$. If $|\iota|<1(=1)$, they are
space-like(null).\\
(c)~If $\dot{z}=0$, the horizontal condition becomes:
\begin{align}
&\frac{x_{2}}{2}\dot{x_{1}}-\frac{x_{1}}{2}+\dot{y}=0,\label{th 7}\\
&-\frac{x_1^2+x_2^2}{2}=0.\label{th 8}
\end{align}
So $x_{1}=x_{2}=0$, $y=y_0$. The curves degenerate into some points.
There are no causal (time-like, space-like, null) horizontal curves
in this family.

Thus, any two points $P_1(x_1^0,x_2^0,y^0,z^0)$,
$Q_1(x_1,x_2^0,y^1,z^0)$ can be connected by a time-like horizontal
curve if
$y^1=-\frac{x_{1}x_{2}^{0}}{2}+\frac{x_{1}^{0}x_{2}^{0}}{2}+y^{0}.$
Especially, any two points $(x_1^0,0,y^0,z^0)$, $(x_1,0,y^0,z^0)$
can be connected by a time-like horizontal straight line.

Any two points $P_1(x_1^0,x_2^0,y^0,z^0)$, $Q_2(x_1^0,x_2,y^1,z^1)$
can be connected by a space-like horizontal curve if
$y^1=\frac{x_{1}^0}{2}x_2+y^0$,
$z^1=\frac{1}{6}x_2^3+\frac{\left(x_1^0\right)^2}{2}x_2+z_0.$

Any two points $P_1(x_1^0,x_2^0,y^0,z^0)$, $Q_3(x_1,x_2,y^0,z^1)$
can be connected by a time-like(space-like, null) horizontal curve
if $x_1=\iota x_2$, $z^1=
\frac{1}{6}\left(1+\iota^{2}\right)x_{2}^{3}+z^{0},$ and
$\mid\iota\mid=\left|\frac{x_{1}^{0}}{x_{2}^{0}}\right|>1(<1,\ =1).$

\section{\normalsize Sub-Lorentzian
geodesics} In the Lorentzian geometry there are no curves of minimal
length because two arbitrary points can be connected by a piecewise
light-like curve whose length is always $0$. For example, let
$\mathbb{R}^2$ be the two dimensional Minkowski space,
$\hat{p}=(\hat{x},\hat{y})$ is any one point in this space. We want
to find a light-like curve going from the origin to $\hat{p}$.
First, we choose a curve $\gamma_1(t):(x(t),y(t))=(t,t)$ which
connects the origin and the point $(\frac{\hat x+\hat
y}{2},\frac{\hat x+\hat y}{2})$; then we choose the second curve
$\gamma_2(t):(x(t), y(t))=(t,-t+\hat x+\hat y)$ which joints
$(\frac{\hat x+\hat y}{2},\frac{\hat x+\hat y}{2})$ and $\hat{p}$.
It is easy to check that the curve $\gamma(t)$ consisting of
$\gamma_1$ and $\gamma_2$ is a light-like curve. It goes from the
origin to the point $\hat{p}$, and the length is 0. However, there
do exist time-like curves with maximal length which are time-like
geodesics \cite{O'Neill}. Upon this reason, we will study the
optimality of time-like geodesics, and compute the longest curve
among all horizontal time-like ones on the sub-Lorentzian Engel
group. The computation will be given by extremizing the action
integral ${S}=\frac 1 2\int(-u_1^2+u_2^2)dt$ under constraint
(\ref{h condition}). By Proposition \ref{left}, horizontal time-like
curves are left invariant, so we can assume that the initial point
is origin, i.e., $x_1(0)=x_2(0)=y(0)=z(0)=0$, and time-like initial
velocity is $-u_1^{2}(0)+u_2^{2}(0)=-1$.

Let $\xi=(\xi_1,\xi_2,\xi_3,\xi_4)$ be the vector of costate
variables, so the Hamiltonian function of Pontryagin's maximum
principle is
\begin{align}
H(\xi_0,\xi,q,u)=\xi_0\frac{-u^2_1+u^2_2}{2}+\xi_1u_1+\xi_2u_2+\xi_3\frac{x_1u_2-x_2u_1}{2}+\xi_4\frac{x_1^2+x_2^2}{2}u_2.
\end{align}
where $\xi_0$ is a constant equals to 0 or $-1$. Also, we get the
Hamiltonian system:
\begin{align}
\nonumber&\dot x_1=H_{\xi_1}=u_1, \ \ \dot x_2=H_{\xi_2}=u_2, \ \
\dot y=H_{\xi_3}=\frac{x_1u_2-x_2u_1}{2}, \ \ \dot
z=H_{\xi_4}=\frac{x_1^2+x_2^2}{2}u_2,\\
&\dot \xi_1=-H_{x_1}=-\frac{\xi_3u_2}{2}-\xi_4x_1u_2,\ \
\dot\xi_2=-H_{x_2}=\frac{\xi_3u_1}{2}-\xi_4x_2u_2,\ \
\dot\xi_3=\dot\xi_4=0,
\end{align}
and the maximum condition:
\begin{align}\label{max pri}
H(\xi_0,\xi(t),q(t),u(t))=\max_{\tilde{u}\in\mathbb{R}^2}H(\xi_0,\xi(t),\tilde{q}(t),\tilde{u}),\
\ \xi_0\leq 0,
\end{align}
where ${u}(t)$ is the optimal control, and $(\xi_0,\xi(t))\neq 0$.

\subsection{\normalsize Abnormal extremal trajectories}
We shall investigate the abnormal case $\xi_0=0$. From the maximum
condition (\ref{max pri}) we obtain
\begin{align}
&H_{u_1}=\xi_1-\frac{\xi_3x_2}{2}=0,\label{H_u1}\\
&H_{u_2}=\xi_2+\frac{\xi_3x_1}{2}+\frac{\xi_4(x_1^2+x_2^2)}{2}=0.\label{H_u2}
\end{align}

Differentiating equations (\ref{H_u1}) and (\ref{H_u2}), we obtain
\begin{align}
&0=\dot\xi_1-\frac{\xi_3\dot
x_2}{2}=\dot\xi_1-\frac{\xi_3u_2}{2}=-u_2(\xi_3+\xi_4x_1),\\
&0=\dot\xi_2+\frac{\xi_3\dot x_1}{2}+\xi_4(x_1\dot x_1+x_2\dot
x_2)=u_1(\xi_3+\xi_4x_1).
\end{align}

For the time-like curve, we assume that $-u_1^2+u_2^2=-1$, so
$\xi_3+\xi_4x_1=0$. If $\xi_4=0$, then $\xi_3=0$, and therefore
$\xi=0$. It is a contradiction with the nontriviality of the costate
variables, hence $\xi_4\neq 0$. In this case,
$x_1=\frac{-\xi_3}{\xi_4}$ is a constant, and $u_1=0,\ u_2=\pm i,$
so there is no time-like abnormal extremal in the Engel group $E$.

For the space-like curve, we assume that $-u_1^2+u_2^2=1$, by using
the same method, we get that $u_1=0,\ u_2=\pm 1,$ so the space-like
abnormal extremal trajectories are given by the following
expression:
\begin{align}\label{abnormal extremal}
\gamma(s)=\left(0, \pm  s,0,\pm \frac{s^3}{6}\right).
\end{align}

For the null curve, suppose that $-u_1^2+u_2^2=0$, we can easily get
that $u_1=0,\ u_2=0,$ so the null abnormal extremal trajectories are
trivial curves.

\subsection{\normalsize
Normal geodesics}
\subsubsection{\normalsize Normal Hamiltonian
system} Now we look at the normal case $\xi_0=-1$. It follows from
the
 maximum condition (\ref{max pri}) that
$H_{u_1}=H_{u_2}=0$. Hence
\begin{align}
u_1=-\left(\xi_1-\frac{x_2\xi_3}{2}\right),\ \ \ \
u_2=\xi_2+\frac{\xi_3x_1}{2}+\frac{\xi_4(x_1^2+x_2^2)}{2}.
\end{align}

Let $\zeta_i=(\xi,X_i),i=1,2,$ be the Hamiltonian corresponding to
the basis vector fields $X_1, X_2$ in the cotangent space $T_q^*E$.
They are linear on the fibers of the cotangent space $T^*E$, and
\begin{align}
\zeta_1=\xi_1-\frac{x_2}{2}\xi_3,\ \ \ \
\zeta_2=\xi_2+\frac{x_1}{2}\xi_3+\frac{x_1^2+x_2^2}{2}\xi_4.
\end{align}
So $u_1=-\zeta_1$ and $u_2=\zeta_2$.

The Hamiltonian system in the normal case becomes:
\begin{equation}\label{normal equtions}
\left\{
\begin{array}{ll}
\dot x_1 =\frac{\partial H}{\partial \xi_1}=-(\xi_1-\frac{x_2}{2}\xi_3)=-\zeta_1,\\
\dot x_2=\frac{\partial H}{\partial \xi_2}=(\xi_2+\frac{x_1}{2}\xi_3+\frac{x_1^2+x_2^2}{2}\xi_4)=\zeta_2, \\
\dot y=\frac{\partial H}{\partial
\xi_3}=\zeta_1\frac{x_2}{2}+\zeta_2\frac{x_1}{2}=\frac 1
2(x_1\zeta_2+x_2\zeta_1),\\
\dot z=\frac{\partial H}{\partial \xi_4}=\frac{x_1^2+x_2^2}{2}\zeta_2,\\
\dot \xi_1=-\frac{\partial H}{\partial
x_1}=-\zeta_2(\frac{\xi_3}{2}+x_1\xi_4),\\
\dot \xi_2=-\frac{\partial H}{\partial x_2}=-\frac 1
2\xi_3\zeta_1-x_2\xi_4\zeta_2,\\
\dot\xi_3=-\frac{\partial H}{\partial y}=0,\\
\dot\xi_4=-\frac{\partial H}{\partial z}=0.
\end{array}
\right.
\end{equation}

\begin{definition}
A normal geodesic in the sub-Lorentzian manifold $(E,D,g)$ is a
curve $\gamma:[a,b]\rightarrow E$ that admits a lift
$\Gamma:[a,b]\rightarrow T^*M$, which is a solution of the
 Hamiltonian equations (\ref{normal equtions}). In this case, we say that
$\Gamma$ is a normal lift of $\gamma$.
\end{definition}

Associate with the expression of $H$, a sub-Lorentzian geodesic is
time-like if $H<0$; space-like if $H>0$; light-like if $H=0$.
\begin{remark}
{In fact, abnormal extremal trajectories (\ref{abnormal extremal})
are also normal geodesics, since we can choose the costate variables
as $\widetilde{\xi}=(0, \pm 1,0,0)$, it is easy to check that
$\Gamma(t)=(\gamma,\widetilde{\xi})$ satisfies Hamiltonian equation
(\ref{normal equtions}).} This example also confirms that normal
geodesics and abnormal trajectories are sometimes not mutually
exclusive.
\end{remark}
\begin{lemma} The causal character of normal
sub-Lorentzian geodesics does not depend on time.
\end{lemma}
\begin{proof} The Hamiltonian $H$ is an
integral of the Hamiltonian system, i.e., $\dot H(s)=0$, this
implies that the causality character does not change for all $t\in
[0,\infty)$.
\end{proof}
\begin{remark} If $\gamma(t)$ is a nspc.
normal geodesic on the Engel group, then the orientation will not
change along the curve. In fact, if $\gamma(t)$ is time-like, and it
is future directed at $t=0$, then we have $-u_1^2(t)+u_2^2(t)=-1,$
$u_1(0)>0.$ We only need to show that $u_1(t)$ will not equal to 0
along the curve $\gamma(t)$. Actually, if there is a $t_1>0,$ such
that $u_1(t_1)=0,$ then we have $u_2^2(t_1)=-1$, it is impossible.
So $u_1(t)$ will not change the sign (since $u_1(t)=-\zeta_1(t)$ is
continuous), and $\gamma(t)$ is future directed along the curve. It
is also true for the other cases.
\end{remark}
\subsubsection{\normalsize Maximality of
short arcs of geodesics}
\begin{definition}
Let $\varphi$ be a smooth function on M, U is an open subset in M,
the horizontal gradient $\nabla_H\varphi$ of $\varphi$ is a smooth
horizontal vector field on $U$ such that for each $p\in U$ and $v\in
H$, $\partial_v\varphi(p)=g(\nabla_H\varphi(p),v)$.
\end{definition}
Locally, we can write
$$\nabla_H\varphi=-(\partial_{X_1}\varphi)X_1+\sum_{i=2}^r(\partial_{X_i}\varphi)X_i.$$

Now we give a  proof that the time-like normal geodesics are locally
maximizing curves on the Engel group.
\begin{proposition}
If $\gamma$ is a t.f.d. (t.p.d.) normal geodesic on the Engel group,
then every sufficiently short subarc of $\gamma$ is a maximizer.
\end{proposition}
\begin{proof}
Assume that $\gamma:(a,b)\rightarrow E$ is parameterized by
arc-length,
$\dot\gamma(t)=u_1^0(t)X_1(\gamma(t))+u_2^0(t)X_2(\gamma(t)),$ $X_1$
is the time orientation, and
$\tilde\Gamma(t)=(\gamma(t),\lambda(t))$ is the normal lift of
$\gamma.$ So we have $H(\gamma(t),\lambda(t))=-\frac{1}{2},\
t\in(a,b).$ For any $c\in(a,b), \epsilon>0,$ let
$J_c=(c-\epsilon,c+\epsilon)\subset(a,b)$ be a neighborhood of $c$.
We will prove that $\gamma|_{J_c}$ is maximal for any $c\in(a,b)$
and small $\epsilon>0$. Since the sub-Lorentzian metric is left
invariant, so we can assume that $\gamma(c)=0,\
\lambda(c)=\lambda_0.$ Consider an $n-1$ dimensional hypersurface
$S$ passing through the origin $0$, and satisfying
$\lambda_0(T_0(S))=0.$ Let $\bar{\lambda}$ be a smooth one-form on
an open neighborhood $\Omega$ of $0,$ such that
$\bar{\lambda}(0)=\lambda_0$, and $\forall p\in S\cap\Omega,$
$\bar{\lambda}(p)(T_pS)=0$, $H(p,\bar{\lambda}(p))=-\frac{1}{2}$.
Let $\Gamma_p=(\gamma_p,\lambda_p)$ be the solution of
$\dot\Gamma(t)=\vec{H}(\Gamma(t)),\ \Gamma(c)=(p,\bar\lambda(p))$.
Then clearly $\Gamma_0=\tilde\Gamma.$ Since $\dot\gamma(0)\not\in
T_0S,$ by the Implicit Function Theorem, there exits a
diffeomorphism:
$$\nu:(c-\epsilon,c+\epsilon)\times W\rightarrow U\subset E,$$
$$(t,p)\rightarrow \gamma_p(t),$$
where $W$ is a neighborhood of $0$ in $S$, $U\subset\Omega$ is a
neighborhood of $0$ in $E$. Define a smooth function $V:U\rightarrow
R$ as:
$$V(x)=t,\ \ \ \ if\  x=\gamma_p(t),$$
we will show that $\|\nabla_HV\|=1.$ For this purpose, let $Y_1$ be
the vector field on $U$ given by
$$Y_1(x)=\dot{\gamma}_p(t)=u_1(p,t)X_1(\gamma_p(t))+u_2(p,t)X_2(\gamma_p(t)), \ \ \ \ if \ x=\gamma_p(t),$$
where $u_1(p,t), u_2(p,t)$ are smooth functions on
$W\times(c-\epsilon,c+\epsilon),$ and $u_1(0,t)=u_1^0(t),
u_2(0,t)=u_2^0(t)$. Since $H(p,\bar{\lambda}(p))=-\frac{1}{2}$, by
the construction of $\Gamma_p(t)$, we have
$H(\gamma_p(t),\lambda_p(t))=-\frac{1}{2}$, and $-u_1^2+u_2^2=-1.$
It is easy to check that $Y_1=u_1X_1+u_2X_2, Y_2=u_2X_1+u_1X_2$ is
also an orthonormal basis of $D$, so $\partial_{Y_1}V=1,\
\partial_{Y_2}V=0.$ Therefore, $\nabla_HV=-Y_1,$ $\|\nabla_HV\|=\sqrt{|g(-Y_1,-Y_1)|}=\sqrt{|-u_1^2+u_2^2|}=1.$
Choose $t_1,t_2$ in the domain of $\gamma.$ If $\gamma(t)$ is a
t.f.d. geodesic, then $|u_1^0|>|u_2^0|$, and $u_1^0>0.$ Since
$u_1(0,t)=u_1^0,$ and $u_1(p,t)$ is a smooth function, so there
exists a neighborhood $W_1\times(c-\epsilon_1,c+\epsilon_1)\subset
W\times(c-\epsilon,c+\epsilon)$ such that $u_1(p,t)>0$. Thus
$\nabla_HV=-Y_1$ is past directed. On the other hand, since
$-u_1^2+u_2^2=-1$, we have $|u_1|>|u_2|$. Let
$\eta:[0,\alpha]\rightarrow U$ be a t.f.d. curve with
$\eta(0)=\gamma(t_1)$, $\eta(\alpha)=\gamma(t_2)$, and
$\dot\eta=v_1X_1+v_2X_2,$ then $|v_1|>|v_2|, v_1>0,$ so
$g(\dot\eta,\nabla_HV)=u_1v_1-u_2v_2>0,$ and
\begin{align}
\nonumber L(\gamma|_{[t_1,t_2]})&=t_2-t_1=V(\gamma(t_2))-V(\gamma(t_1))=\int_0^{\alpha}\frac{dV(\eta(s))}{ds}ds\\
\nonumber&=\int_0^{\alpha}g(\dot\eta,\nabla_HV)ds\geq\int_0^{\alpha}\|\dot\eta(s)\|ds=L(\eta|_{[0,\alpha]}).
\end{align}
By the reverse Schwartz inequality, $L(\gamma)=L(\eta)$ holds if and
only if $\eta$ can be reparameterized as a trajectory of
$-\nabla_HV.$ If $\gamma(t)$ is a t.p.d. geodesic, then
$|u_1^0|>|u_2^0|$, and $u_1^0<0.$ By the same method, we choose a
neighborhood such that $W_2\times(c-\epsilon_2,c+\epsilon_2)\subset
W\times(c-\epsilon,c+\epsilon)$ such that $u_1(p,t)<0$. Thus
$\nabla_HV=-Y_1$ is future directed. Let $\rho:[0,\alpha]\rightarrow
U$ be a t.p.d. curve with $\rho(0)=\gamma(t_1),$
$\rho(\alpha)=\gamma(t_2)$, and $\dot\rho=\mu_1X_1+\mu_2X_2,$ then
$|\mu_1|>|\mu_2|, \mu_1<0,$ so
$g(\dot\rho,\nabla_HV)=u_1\mu_1-u_2\mu_2>0,$ and
\begin{align}
\nonumber L(\gamma|_{[t_1,t_2]})&=t_2-t_1=V(\gamma(t_2))-V(\gamma(t_1))=\int_0^{\alpha}\frac{dV(\eta(s))}{ds}ds\\
\nonumber&=\int_0^{\alpha}g(\dot\rho,\nabla_HV)ds\geq\int_0^{\alpha}\|\dot\rho(s)\|ds=L(\rho|_{[0,\alpha]}).
\end{align}
By the reverse Schwartz inequality, $L(\gamma)=L(\rho)$ holds if and
only if $\rho$ can be reparameterized as a trajectory of
$-\nabla_HV.$ In conclusion, the t.f.d(t.p.d.) normal geodesics are
locally maximizers. This ends the proof.
\end{proof}
Next, we compute the expressions of light-like geodesics and
time-like geodesics on the Engel group.

Differentiating $\zeta_i$,
\begin{align}
&\dot \zeta_1=\dot\xi_1-\frac{\xi_3}{2}\dot x_2=-\zeta_2(\xi_3+x_1\xi_4),\\
&\dot\zeta_2=\dot\xi_2+\frac{1}{2}\dot{x}_{1}\xi_{3}+(x_{1}\dot{x}_{1}+x_{2}\dot{x}_{2})\xi_{4}
  =-\zeta_1(\xi_3+x_1\xi_4).
\end{align}
Let
\begin{align}
\beta(s)=-(\xi_3+x_1\xi_4), \ \ \dot\beta=\xi_4\zeta_1,
\end{align}
then we have
\begin{align}\label{nor equations2}
\dot\zeta_1=\beta\zeta_2,\ \ \dot\zeta_2=\beta\zeta_1,\ \
\dot\beta=\xi_4\zeta_1.
\end{align}
\subsubsection{\normalsize Light-like
geodesics} Firstly, we study the case of light-like sub-Lorentzian
geodesics.

By the definition, we have $H=\frac{1}{2}(-h_{1}^{2}+h_{2}^{2})=0$,
thus $h_{2}=\pm h_{1}.$ If $h_{2}=h_{1}$, then light-like
trajectories satisfy the ODE:
$$\dot{\gamma}=-h_{1}(X_{1}-X_{2}),$$
i.e. they are reparameterizations of the one-parametric subgroup of
the field $X_{1}-X_{2}$. We assume $\dot{\gamma}=X_{1}-X_{2}$, so
\begin{align}
\dot{x}_{1}=1,\ \dot{x}_{2}=-1,\ \dot{y}=-\frac{1}{2}(x_{1}+x_{2}),\
\dot{z}=-\frac{1}{2}(x_{1}^{2}+x_{2}^{2}),\nonumber
\end{align}
thus
\begin{align}
x_{1}=t,\ x_{2}=-t,\ y=0,\ z=-\frac{1}{3}t^{3}.\nonumber
\end{align}
If $h_{2}=-h_{1}$, similarly, we obtain
\begin{align}
x_{1}=t,\ x_{2}=t,\ y=0,\ z=\frac{1}{3}t^{3}.\nonumber
\end{align}
In conclusion, we get the following theorem:
\begin{theorem}
Light-like horizontal geodesics starting from the origin
 are reparameterizations of the curves:
\begin{align}
x_{1}=t,\ x_{2}=\pm t,\ y=0,\ z=\pm\frac{1}{3}t^{3},\nonumber
\end{align}
i.e., they are reparameterizations of the one-parameter subgroups
corresponding to the vector fields $X_1\pm X_2.$
\end{theorem}

\subsubsection{\normalsize Time-like
geodesics}Secondly, we study time-like sub-Lorentzian geodesics on
the Engel group.

We consider the case of $\xi_4=0$ at first. This case is also of
interest since it reproduces the earlier known results for the
Heisenberg group \cite{Korolko.A}. In this case
$\beta=-(\xi_3+x_1\xi_4)=-\xi_3$ is a constant. Equations (\ref{nor
equations2}) become
\begin{align}\label{ode0}
\dot\zeta_1=-\xi_3\zeta_2,\ \ \dot\zeta_2=-\xi_3\zeta_1,
\end{align}
where $\xi_3$ is a constant. There are two separate cases:

\textbf{Case 1}: If $\xi_3=0$, we have $\zeta_1$ and $\zeta_2$ are
constants, i.e., $\zeta_1(s)=\zeta_1(0)=\xi_1(0)$ and
$\zeta_2(s)=\zeta_2(0)=\xi_2(0)$. According to (\ref{normal
equations}), $\xi_1$ and $\xi_2$ are constants. On the other hand,
by integrating $\dot x_1=-\zeta_1$ and $\dot x_2=\zeta_2$, we get
\begin{align}
x_1(s)=-\xi_1s \ \  \hbox{and}\ \ x_2(s)=\xi_2s.
\end{align}
Since $\dot y=\frac 1 2(x_1\zeta_2+x_2\zeta_1)=0$, then $y(s)=0$.
Also
$$\dot z=\frac{x_1^2+x_2^2}{2}\zeta_2=\frac{\xi_1^2+\xi_2^2}{2}\xi_2s^2,$$
so
$$z(s)=\frac{\xi_1^2+\xi_2^2}{6}\xi_2s^3=\frac{x_1^2(s)x_2(s)+x_2^3(s)}{6}.$$

\begin{theorem}
In the case of $\xi_3=\xi_4=0$, there is a unique time-like
horizontal geodesic joining the origin to a point $(x_1,x_2,y,z)$,
if and only if $y=0$, $z$ is the following function of $x_1,x_2$:
\begin{align}z=\frac{x_1^2x_2+x_2^3}{6}.\end{align}
And the expression of the geodesic is
\begin{align}
x_1(s)=-\xi_1s,\ \ x_2(s)=\xi_2s,\ \ y(s)=0,\ \
z(s)=\frac{\xi_1^2+\xi_2^2}{6}\xi_2s^3,
\end{align}
where $\xi_1, \xi_2$ are constants.
The arc-length is given by
\begin{align}
l=\sqrt{x_1^2-x_2^2}.
\end{align}
Its projection to the $(x_1,x_2)$ plane is a straight line.
\end{theorem}

\textbf{Case 2}: If $\xi_3\neq0$, from (\ref{ode0}), we have
\begin{align}
&\zeta_1(s)=\xi_1^{0}\cosh(\xi_3s)-\xi_2^{0}\sinh(\xi_3s),\\
&\zeta_2(s)=-\xi_1^{0}\sinh(\xi_3s)+\xi_2^{0}\cosh(\xi_3s),
\end{align}
where $\xi_1^{0}=\xi_1(0)$, $\xi_2^{0}=\xi_2(0)$. So
\begin{align}
&x_1=-\int_{0}^s\zeta_1(t)dt=-\frac{\xi_1^{0}}{\xi_3}\sinh(\xi_3s)+\frac{\xi_2^{0}}{\xi_3}\left(\cosh(\xi_3s)-1\right),\label{x 1}\\
&x_2=\int_{0}^t\zeta_2(t)dt=-\frac{\xi_1^{0}}{\xi_3}\left(\cosh(\xi_3s)-1\right)+\frac{\xi_2^{0}}{\xi_3}\sinh(\xi_3s).\label{x
 22}
\end{align}

Substituting them into the expression of $\dot y, \dot z$ in
(\ref{normal equtions}), and integrating, we get
\begin{theorem}
In the case of $\xi_{3}\neq0, \xi_{4}=0$, the time-like horizontal
geodesics starting from the origin are given by:
\begin{align}
&x_1(s)=-A_1\sinh(\xi_3s)+A_2\left(\cosh(\xi_3s)-1\right),\label{x 1}\\
&x_2(s)=-A_1\left(\cosh(\xi_3s)-1\right)+A_2\sinh(\xi_3s),\label{x 2}\\
&y(s)=\frac{1}{2}(A_2^2-A_1^2)\left(\xi_3s-\sinh(\xi_3s)\right),\label{y(s)}\\
&z(s)=A_2(A_1^2+A_2^2)\cosh^2(\xi_3s)\sinh(\xi_3s)
-\frac{2}{3}A_2^3\sinh^{3}(\xi_3s)-\frac{1}{3}A_1(A_1^2+3A_2^2)\cosh^3(\xi_3s)\nonumber\\
&\ \ \ \ +\frac{1}{2}A_1(A_1^2+3A_2^2)\cosh^{^{2}}(\xi_3s)-\frac{1}{2}A_2(3A_1^2+A_2^2)\sinh(\xi_3s)\cosh(\xi_3s)-\frac{1}{2}A_2(3A_1^2+A_2^2){s}\nonumber\\
&\ \ \ \ -\frac{1}{6}A_1(A_1^2+3A_2^2).
\end{align}
where $\xi_1^0=\xi_1(0), \xi_2^0=\xi_2(0)$ is the initial value,
$\xi_3$, $A_1=\frac{\xi_1^0}{\xi_3}$, $A_2=\frac{\xi_2^0}{\xi_3}$
are constants.
\end{theorem}
Projections of geodesics to the plane $(x_1,x_2)$ are hyperbolas,
for $\xi(0)=(\sqrt{2},1,1,0)$, \
$\xi(0)=(\frac{\sqrt{5}}{2},\frac{1}{2},1,0)$ and
$\xi(0)=(\frac{\sqrt{5}}{2},\frac{1}{2},-1,0),$ they are shown in
Figure \ref{fig2}.
\begin{figure}[h]
\begin{minipage}[t]{0.3\linewidth}
\centering
\includegraphics[width=0.8\textwidth]{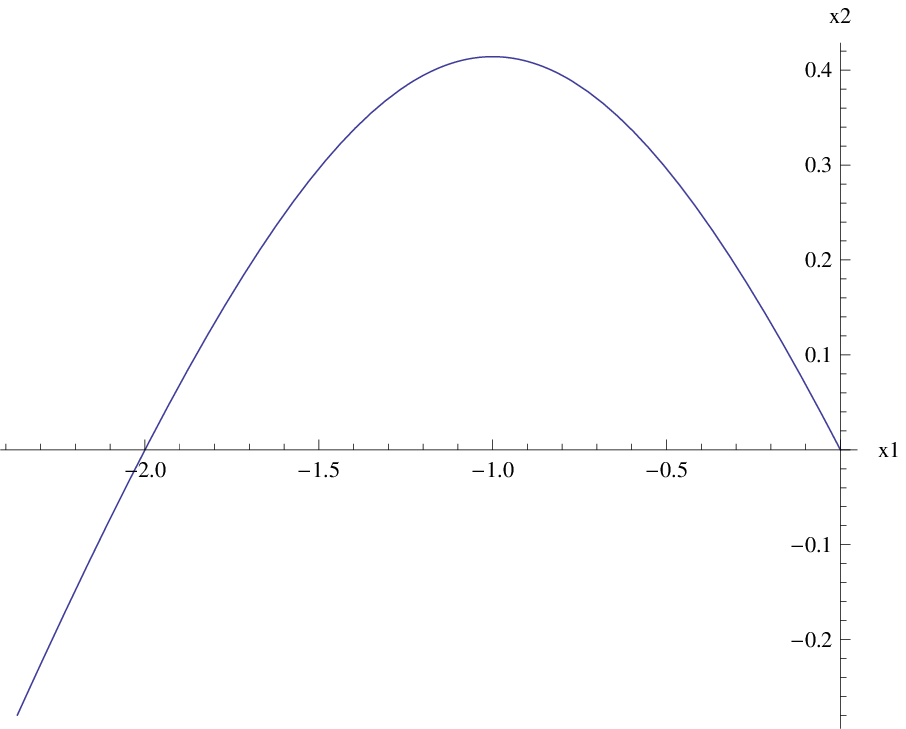}
\end{minipage}
\begin{minipage}[t]{0.3\linewidth}
\centering
\includegraphics[width=0.8\textwidth]{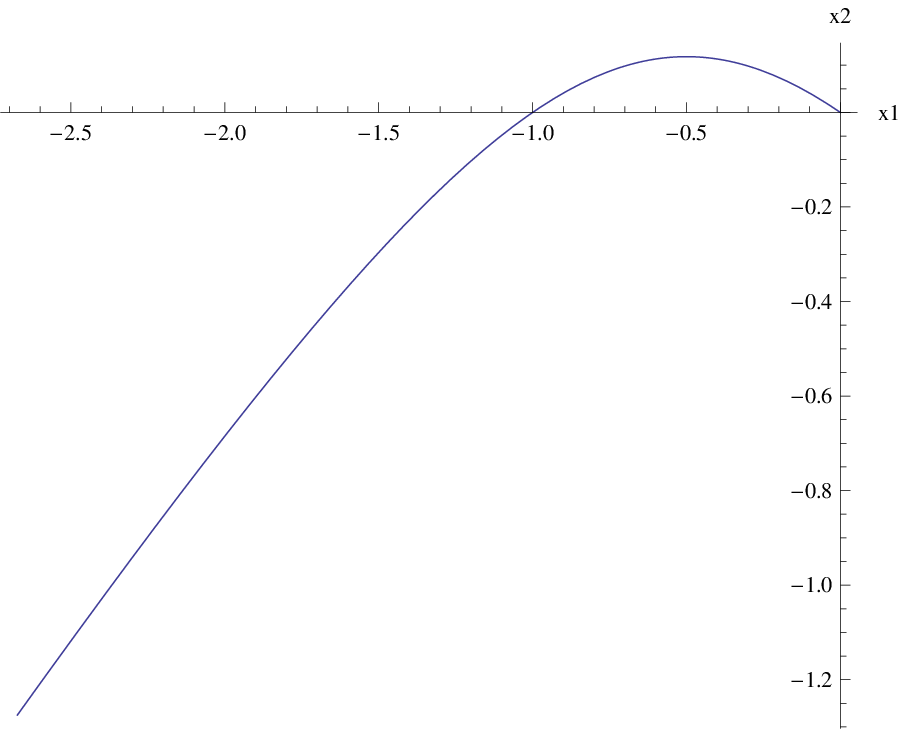}
\end{minipage}
\begin{minipage}[t]{0.3\linewidth}
\centering
\includegraphics[width=0.8\textwidth]{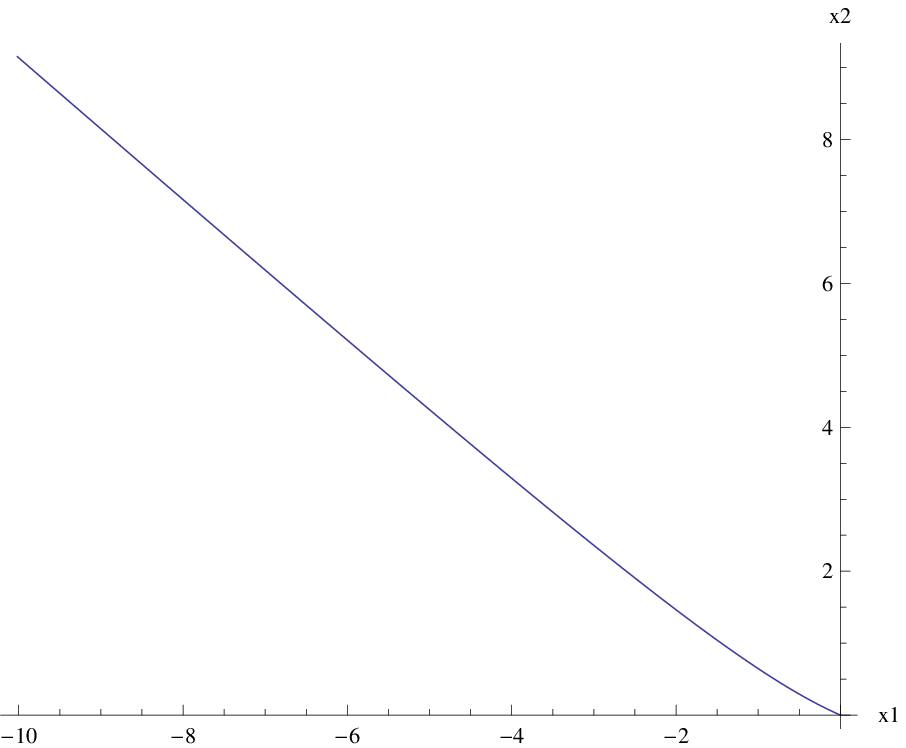}
\end{minipage}
\caption{Projections of geodesics to the plane $(x_1,x_2)$ when
$\xi_3\neq0,\ \xi_4=0$\label{fig2}}
\end{figure}

From this theorem, we obtain a description of the reachable set by
geodesics $\xi_{3}\neq0, \xi_{4}=0$ starting from the origin.
\begin{corollary}
In the case of $\xi_{3}\neq0, \xi_{4}=0$, let $(x_1,x_2,y,z)$ be a
point on a time-like geodesic, then we have
$$-1<\frac{4y}{-x_1^2+x_2^2}<1.$$
\end{corollary}
\begin{proof}
 By (\ref{x 1}) and (\ref{x 2}), we get
\begin{align}\label{x_1^2}
-x_1^2+x_2^2=4(A_2^2-A_1^2)\sinh^2\left(\frac{\xi_3}{2}\right),
\end{align}
substituting (\ref{x_1^2}) into (\ref{y(s)}), we obtain the
following equation:
\begin{align}
y=\frac{(-x_1^2+x_2^2)(\xi_3-\sinh(\xi_3))}{8\sinh^2\left(\frac{\xi_3}{2}\right)},
\end{align}
if we set $\tau=\frac{\xi_3}{2}$, then
\begin{align}
y=\frac{(-x_1^2+x_2^2)}{4}\left(\frac{\tau}{\sinh^2(\tau)}-{\coth(\tau)}\right),
\end{align}
or
\begin{align}\label{re1}
\frac{4y}{(-x_1^2+x_2^2)}=\frac{\tau}{\sinh^2(\tau)}-{\coth(\tau)}.
\end{align}
It is easy to check that the right hand side of (\ref{re1}) is a
decreasing function in $(-\infty,+\infty),$ and its range is
$(-1,1).$ That is to say, the points on the time-like geodesics
should satisfy $$-1<\frac{4y}{-x_1^2+x_2^2}<1.$$ This ends the
proof.
\end{proof}

Next, we consider the case $\xi_4\neq0$. Recall that
\begin{align}\label{normal equations2}
\dot\zeta_1=\beta\zeta_2,\ \ \dot\zeta_2=\beta\zeta_1,\ \
\hbox{where} \ \ \beta(s)=-(\xi_3+x_1\xi_4),\ \
\dot\beta=\xi_4\zeta_1.\
\end{align}
Combining the expressions for $\dot \beta$ and $\dot \zeta_2$ to get
\begin{align}
\xi_4\dot\zeta_2=\beta\xi_4\zeta_1=\beta\dot\beta.
\end{align}
Integrating both sides, we have
\begin{align}
\xi_4\zeta_2=\frac{\beta^2}{2}+C_1,\ \ \hbox{where}\ \
C_1=\xi_4\zeta_2(0)-\frac{\beta^2(0)}{2}=\xi_4\xi_2^{0}-\frac{\xi_3^2}{2}.
\end{align}
This yields
\begin{align}
x_1(s)=-\frac{\beta(s)+\xi_3}{\xi_4},\label{eq 1}
\end{align}
and
\begin{align}
\zeta_2(s)=\frac{1}{\xi_4}\left(\frac{\beta^2(s)}{2}+C_1\right).
\end{align}
Since $\dot x_2=\zeta_2,$ we deduce
\begin{align}
x_2(s)=\int_0^s\zeta_2(t)dt=\frac{1}{\xi_4}\int_0^s\left(\frac{\beta^2(t)}{2}+C_1\right)dt.\label{eq
2}
\end{align}

To compute $y(s)$ in term of $\beta(s)$, we note that
\begin{align}
\dot y=\frac 1 2(x_1\zeta_2+x_2\zeta_1)=\frac12(x_1\dot x_2-x_2\dot
x_1),\label{eq 3}
\end{align}
then integration by parts yields
\begin{align}
y(s)&=\frac12\int_0^s(x_1\dot x_2-x_2\dot
x_1)dt=\int_0^sx_1\zeta_2dt-\frac12x_1x_2\nonumber\\
&=-\frac{1}{\xi_4^2}\int_0^s\left({\beta(t)+\xi_3}\right)\left(\frac{\beta^2(t)}{2}+C_1\right)dt-\frac12x_1x_2.\label{eq
33}
\end{align}
Finally, since $\dot z=\frac{x_1^2+x_2^2}{2}\zeta_2,$
\begin{align}
z(s)&=\int_0^s\frac{x_1^2+x_2^2}{2}\zeta_2dt=\frac12\int_0^sx_1^{2}\zeta_2dt+\frac16x_2^3\nonumber\\
&=\frac{1}{2\xi_4^3}\int_0^s\left({\beta(t)+\xi_3}\right)^2\left(\frac{\beta^2(t)}{2}+C_1\right)dt+\frac16x_2^3.\label{eq
4}
\end{align}
Once we find $\beta$, we can find the geodesic $(x_1(s), x_2(s),
y(s), z(s))$ explicitly.

Since $\dot \beta(s)=\xi_4\zeta_1$,
$\dot\beta(0)=\xi_4\zeta_1(0)=\xi_4\xi_1^{0}$, we have
\begin{align}
\ddot\beta(s)=\xi_4\dot\zeta_1=\xi_4\beta(s)\zeta_2=\beta(s)(\xi_4\zeta_2)=\beta(s)\left(\frac{\beta^2(s)}{2}+C_1\right).
\end{align}
Multiplying both sides by $2\dot\beta(s)$ and integrating, we have
\begin{align}
\dot\beta^2(s)=\frac{\beta^4(s)}{4}+C_1\beta^2(s)+C_2=\left(\frac{\beta^2(s)}{2}+C_1\right)^2+C_2-C_1^2,
\end{align}
where $C_2$ is a constant, and
\begin{align}
C_2=\dot\beta^2(0)-\frac{\beta^4(0)}{4}-C_1\beta^2(0)=(\xi_1^{0})^2\xi_4^2+\frac{\xi_3^4}{4}-\xi_2^{0}\xi_3^2\xi_4.
\end{align}
Then
\begin{align}\label{causal}
C_2-C_1^2=(\xi_1^{0})^2\xi_4^2+\frac{\xi_3^4}{4}-\xi_2^{0}\xi_3^2\xi_4-\left(\xi_2^{0}\xi_4-\frac{\xi_3^2}{2}\right)^2=\xi_4^2((\xi_1^{0})^2-(\xi_2^{0})^2)=\xi_4^2,
\end{align}
since $(\xi_1^{0})^2-(\xi_2^{0})^2=1$.

Assume $\dot\beta(s)>0$, we have
\begin{align}
\frac{d\beta(s)}{ds}=\sqrt{\left(\frac{\beta^2(s)}{2}+C_1\right)^2+\xi_4^2}.
\end{align}
Hence
\begin{align}\label{ode}
ds=\frac{d\beta}{\sqrt{\left(\frac{\beta^2(s)}{2}+C_1\right)^2+\xi_4^2}}.
\end{align}
Let $\rho^2=C_1+\xi_4i$, $\bar{\rho}^2=C_1-\xi_4i$ and
$u=\frac{\beta}{\sqrt 2}$, integrating (\ref{ode}) from $0$ to $s$,
we obtain
\begin{align}
s=\int_{\frac{\beta(0)}{\sqrt 2}}^{\frac{\beta(s)}{\sqrt
2}}\frac{\sqrt 2du}{\sqrt{(u^2+\rho^2)(u^2+\bar\rho^2)}}.
\end{align}

Set
\begin{align}
&k^2=-\frac{(\rho-\bar\rho)^2}{4\rho\bar\rho}=\frac{\sqrt{C_1^2+\xi_4^2}-C_1}{2\sqrt{C_1^2+\xi_4^2}},\nonumber\\
&\mathfrak{g}=\frac{1}{2\sqrt{\rho\bar\rho}}=\frac{1}{2(C_1^2+\xi_4^2)^{\frac{1}{4}}}.\nonumber
\end{align}
Since
\begin{align}\label{Jacobi function}
\int_y^\infty\frac{dt}{\sqrt{(t^2+\rho^2)(t^2+\bar\rho)}}=\mathfrak{g}\cdot
cn^{-1}\ (cos\ \varphi,k)=\mathfrak{g}F(\varphi,k),
\end{align}
where $cn^{-1}\ (y,k)$ is a Jacobi's Inverse Elliptic Functions, and
\begin{align}
\varphi=\cos^{-1}\
\left(\frac{y^2-\rho\bar\rho}{y^2+\rho\bar\rho}\right),\ \ \
F(\varphi,k)=\int_0^\varphi\frac{dt}{\sqrt{1-k^2sin^2\ t}}.\nonumber
\end{align}
Hence
\begin{align}
\int_{\frac{\beta(0)}{\sqrt 2}}^{\frac{\beta(s)}{\sqrt
2}}\frac{\sqrt
2du}{\sqrt{(u^2+\rho^2)(u^2+\bar\rho^2)}}=\int_{\frac{\beta(0)}{\sqrt
2}}^{\infty}\frac{\sqrt
2du}{\sqrt{(u^2+\rho^2)(u^2+\bar\rho^2)}}-\int_{\frac{\beta(s)}{\sqrt
2}}^{\infty}\frac{\sqrt
2du}{\sqrt{(u^2+\rho^2)(u^2+\bar\rho^2)}}.\nonumber
\end{align}
According to (\ref{Jacobi function}), we have
\begin{align}
\int_{\frac{\beta(0)}{\sqrt 2}}^{\infty}\frac{\sqrt
2du}{\sqrt{(u^2+\rho^2)(u^2+\bar\rho^2)}}=\sqrt2\mathfrak{g}F(\varphi_1,k)=constant,
\end{align}
where
\begin{align}\varphi_1=cos^{-1}\ \left(\frac{\xi_3^2-2\sqrt{C_1+\xi_4^2}}{\xi_3^2+2\sqrt{C_1+\xi_4^2}}\right)\nonumber.\end{align}

Since
\begin{align}
\int_{\frac{\beta(s)}{\sqrt 2}}^{\infty}\frac{\sqrt
2du}{\sqrt{(u^2+\rho^2)(u^2+\bar\rho^2)}}=\sqrt2\mathfrak{g}\cdot
cn^{-1}\
\left(\frac{\beta^2(s)-2\rho\bar\rho}{\beta^2(s)+2\rho\bar\rho}\right).
\end{align}
Hence
\begin{align}
cn^{-1}\
\left(\frac{\beta^2(s)-2\rho\bar\rho}{\beta^2(s)+2\rho\bar\rho}\right)=F(\varphi_1,k)-\frac{s}{\sqrt2\mathfrak{g}},
\end{align}
let $F=F(\varphi_1,k),$ we obtain
\begin{align}\label{beita2}
\beta^2(s)=\frac{2\rho\bar\rho\left(1+cn\ \left(F
-\frac{s}{\sqrt2\mathfrak{g}},k\right)\right)}{\left(1-cn\
\left(F-\frac{s}{\sqrt2\mathfrak{g}},k\right)\right)}=
\frac{2\rho\bar\rho\left(1+cn\
\left(2\tilde{s},k\right)\right)}{\left(1-cn\
\left(2\tilde{s},k\right)\right)},
\end{align}
where $2\tilde{s}=F-\frac{s}{\sqrt2\mathfrak{g}}$.

Since
\begin{align}
\frac{1-cn\ (2s)}{1+cn\ (2s)}=tn^2\ (s)dn^2\ (s),
\end{align}
hence
\begin{align}\label{beita1}
\beta(s)=\frac{\sqrt{2\rho\bar\rho}}{tn\ (\tilde s,k)dn\ (\tilde
s,k)}=\sqrt{2\rho\bar\rho}cs\ (\tilde s,k)nd\ (\tilde s,k).
\end{align}

For the case of $\dot{\beta}(s)<0,$ we can calculate by the same
method, and get the same result. But the expression of the parameter
$\tilde{s}$ in (\ref{beita2}) and(\ref{beita1}) should be changed to
$$\frac{1}{2}\left(F+\frac{s}{\sqrt{2}\mathfrak{g}}\right).$$ Thus the
sign of $\dot{\beta}(s)$ will not affect the expression of the
geodesics.

Therefore, integrating equations (\ref{eq 1}), (\ref{eq 2}),
(\ref{eq 33}) and (\ref{eq 4}), we get a complete description of the
Hamiltonian time-like geodesics in the Engel group.
\begin{theorem}\label{normal timelike geodesics}
In the case of $\xi_{4}\neq0$, time-like geodesics starting from the
origin are given by:
\begin{align}
&x_{1}(s)=-\frac{1}{\xi_{4}}(\beta(s)+\xi_3),\\
&x_{2}(s)=\frac{1}{2\xi_4}(B_2(s)+2C_1),\\
&y(s)=-\frac{1}{2\xi_4^2}(B_3(s)+2C_1B_1(s)+\xi_3B_2(s)+2C_1\xi_3s)-\frac{1}{2}x_1(s)x_2(s),\\
&z(s)=\frac{1}{4\xi_4^3}(B_4(s)+2C_1B_2(s)+2\xi_3B_3(s)+4C_1\xi_3B_1(s)+\xi_3^2B_2(s)+2C_1\xi_3^2)+\frac{1}{6}x_2^3(s),
\end{align}
where $C_1=\xi_4\xi_2^0-\frac{\xi_3^2}{2}$,
$B_i(s)=\int^{s}_{0}\beta^{i}(t)dt,\ i=1,\ldots,4,$ and the
expressions of $B_i(s)$ are presented in Appendix.
\end{theorem}
Projections of geodesics to the plane $(x_1,x_2)$ with
$\xi(0)=(1,0,1,1)$, \ $\xi(0)=(\frac{\sqrt{5}}{2},1/2,2,1)$ and
$\xi(0)=(\frac{\sqrt{5}}{2},1/2,1,1)$ are shown in Figure
\ref{fig3}.
\begin{figure}[h]
\begin{minipage}[t]{0.3\linewidth}
\centering
\includegraphics[width=0.8\textwidth]{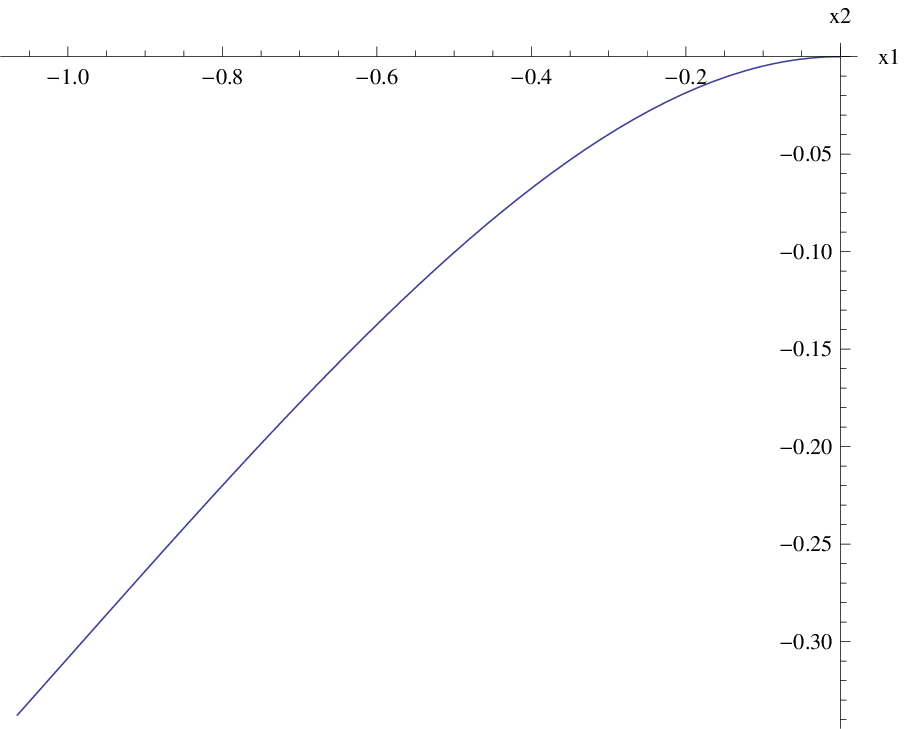}
\end{minipage}
\begin{minipage}[t]{0.3\linewidth}
\centering
\includegraphics[width=0.8\textwidth]{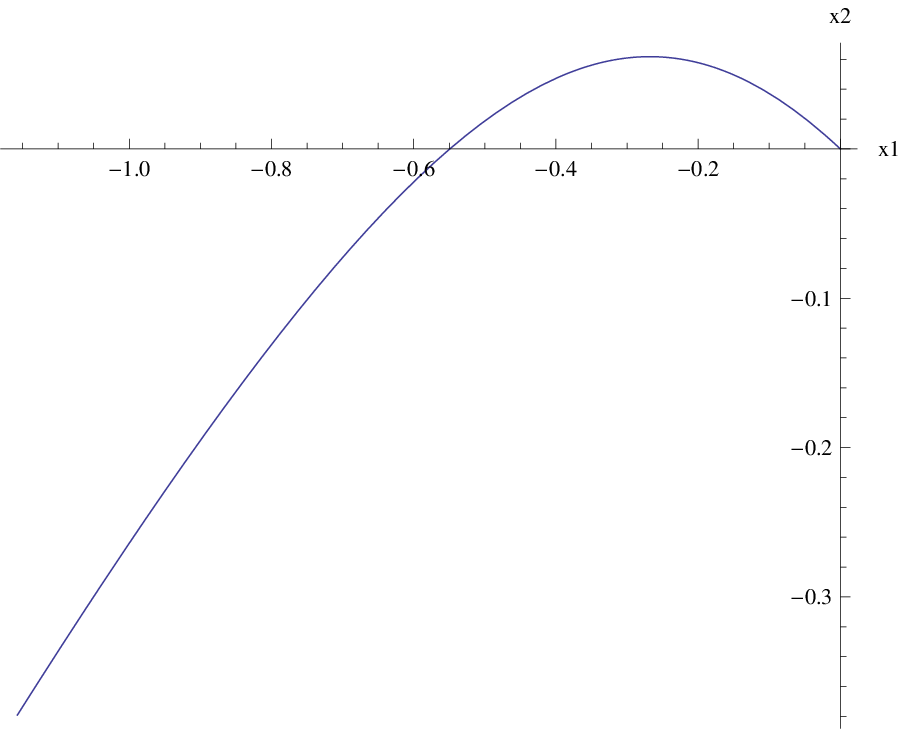}
\end{minipage}
\begin{minipage}[t]{0.3\linewidth}
\centering
\includegraphics[width=0.8\textwidth]{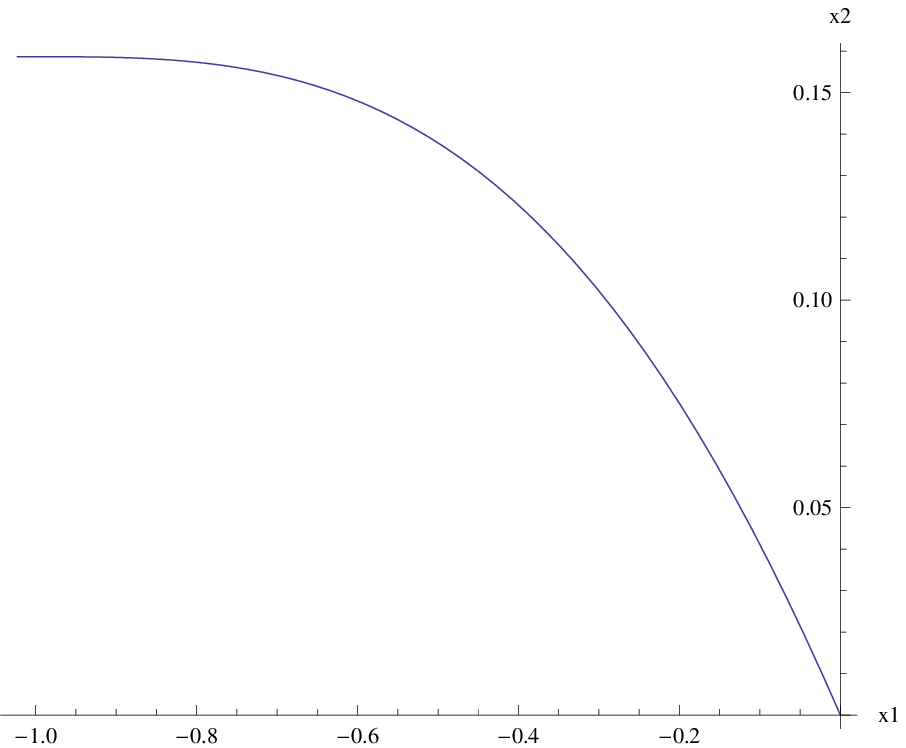}
\end{minipage}
\caption{Projections of geodesics to the plane $(x_1,x_2)$ when
$\xi_3\neq0,\ \xi_4\neq0$\label{fig3}}
\end{figure}

\subsection*{\normalsize Appendix} Denoting
\begin{align}
\nonumber&\psi_1(s)=ln\ (k'^2+cs^2\ (s,k)),
\nonumber&\psi_2(s)=\frac{dn\ (2s,k)sn\ (2s,k)}{(1-cn\ (2s,k))^2},\\
\nonumber&\psi_3(s)=\frac{E(2s,k)-E(2s,k)cn\ (2s,k)}{1-cn\ (2s,k)},
\nonumber&F(\varphi,k)=\int_0^\varphi\frac{dt}{\sqrt{1-k^2sin^2\
t}},
\end{align}
\begin{align}
\nonumber
k^2=\frac{\sqrt{C_1^2+\xi_4^2}-C_1}{2\sqrt{C_1^2+\xi_4^2}}, \ \ \ \
\ \ \mathfrak{g}=\frac{1}{2(C_1^2+\xi_4^2)^{\frac{1}{4}}},\ \ \ \ \
\ \varphi_1=cos^{-1}\
\left(\frac{\xi_3^2-2\sqrt{C_1+\xi_4^2}}{\xi_3^2+2\sqrt{C_1+\xi_4^2}}\right)
\end{align}
we get the expressions of $B_i(s)$ as following:
\begin{align}
\nonumber B_1(s)&=\int^{s}_{0}\beta(t)dt=\mathfrak{g}\psi_1(\tilde{s})+D_{1},\\
\nonumber B_2(s)&=\int^{s}_{0}\beta^{2}(t)dt=\frac{\sqrt{2}}{\mathfrak{g}}[-3\tilde{s}+(1-cn\ (2\tilde{s},k))\psi_2(\tilde{s})+\psi_3(\tilde{s})]+D_2,\\
\nonumber B_3(s)&=\int^{s}_{0}\beta^{3}(t)dt=\frac{1}{2\mathfrak{g}^2k'^2}[k'^2cs^2\ (\tilde{s},k)+k'^2(2k^2-1)\psi_1(\tilde{s})-(2k^4-k^6-k^2)e^{\psi_1(\tilde{s})}]+D_3,\\
\nonumber
B_4(s)&=\int^{s}_{0}\beta^{4}(t)dt=\frac{\sqrt{2}}{3\mathfrak{g}^{3}}[-\frac{3}{2}\tilde{s}+(3-4k^2)\psi_2(\tilde{s})+4k'^2(1+k^2)\tilde{s}-2k'^2E(2\tilde{s},k)]
+D_{4},
\end{align}
where
$\tilde{s}=\frac{1}{2}\left(F\pm\frac{s}{\sqrt{2}\mathfrak{g}}\right)$,
$F=F(\varphi_1,k)$, $D_{1}, D_{2}, D_{3}, D_{4}$ are constants, and
\begin{align}
\nonumber D_1&=-\mathfrak{g}\psi_1\left(\frac{F}{2}\right),\\
\nonumber D_2&=\frac{\sqrt{2}}{\mathfrak{g}}\left[\frac{3F}{2}-(1-cn\ (F,k))\psi_2\left(\frac{F}{2}\right)-\psi_3\left(\frac{F}{2}\right)\right],\\
\nonumber D_3&=\frac{1}{2\mathfrak{g}^2k'^2}\left[-k'^2cs^2\ \left(\frac{F}{2},k\right)-k'^2(2k^2-1)\psi_1\left(\frac{F}{2}\right)+(2k^4-k^6-k^2)e^{\psi_1\left(\frac{F}{2}\right)}\right],\\
\nonumber
D_4&=\frac{\sqrt{2}}{3\mathfrak{g}^{3}}\left[\frac{3F}{4}-(3-4k^2)\psi_2\left(\frac{F}{2}\right)-2k'^2(1+k^2)F+2k'^2E(F,k)\right].
\end{align}

\subsection*{\normalsize Acknowledgement}
This work was supported by NSFC(No.11071119, No.11401531) and
NSFC-RFBR (No. 11311120055). The authors cordially thank the
referees for their careful reading and helpful comments.

\clearpage

\end{document}